\documentclass{birkmult}
\usepackage{amsmath, cite}
\usepackage{amsthm}
\usepackage{amssymb}
\usepackage{amsfonts}
\usepackage{latexsym, multicol}
\usepackage{hyperref}

 \newcommand{\W}{W_{\phi,\psi}}
 
 \newcommand{\h}{\mathcal{H}}

 \newcommand{\C}{\mathbb{C}}
 \newcommand{\D}{\mathbb{D}}

\newcommand{\R}{ \mathbb{R}}

\renewcommand{\Im}{\operatorname{Im}}

\newcommand{\norm}[1]{\| #1 \|}

\renewcommand{\labelenumi}{(\roman{enumi})}
\renewcommand{\phi}{\varphi}

\theoremstyle{plain}
\newtheorem{Theorem}{Theorem}
\newtheorem{Proposition}[Theorem]{Proposition}

\theoremstyle{definition}

\newtheorem{Example}{Example}

\newtheorem{Question}{Question}

\numberwithin{Theorem}{section}
\allowdisplaybreaks

\allowdisplaybreaks

\begin{document}
\bibliographystyle{plain}

\title{Which weighted composition operators are complex symmetric?}
	\author{Stephan Ramon Garcia}
	\address{Department of Mathematics\\Pomona College\\Claremont, California\\91711 \\ USA}
	\email{Stephan.Garcia@pomona.edu}
	\urladdr{\url{http://pages.pomona.edu/~sg064747}}

	\author{Christopher Hammond}
	\address{
	Department of Mathematics\\
	Connecticut College\\
	270 Mohegan Avenue\\
	New London, CT 06320\\USA}
	\email{cnham@conncoll.edu}
	\urladdr{\url{http://math.conncoll.edu/faculty/chammond/}}


    \thanks{First author partially supported by National Science Foundation Grant DMS-1001614.}

	\begin{abstract}
	    	Recent work by several authors has revealed the existence of many unexpected classes of
		normal weighted composition operators.  On the other hand, it is known that
		every normal operator is a complex symmetric operator.  We therefore 
		undertake the study of complex symmetric weighted composition operators, identifying
		several new classes of such operators.  
	\end{abstract}

\maketitle
\thispagestyle{empty}

\section{Introduction}
	In 2010, C.~Cowen and E.~Ko obtained an explicit characterization and spectral description of 
	all hermitian weighted composition operators on the classical Hardy space $H^2$ \cite{CowenKo}.  
	This work was later extended to certain weighted Hardy spaces by 
	C.~Cowen, G.~Gunatillake, and E.~Ko \cite{CGK}.  Along similar lines, P.~Bourdon and S.~Narayan
	have recently studied normal weighted composition operators on $H^2$ \cite{BourdonNarayan}.  
	Taken together, these articles have established the existence of
	several unexpected families of normal weighted composition operators.  	

	It turns out that normal operators are the simplest examples of complex symmetric operators.
	We say that a bounded operator $T$ on a complex Hilbert space $\h$ is \emph{complex symmetric} 
	if there exists a  \emph{conjugation} (i.e., a conjugate-linear, isometric involution) $J$ such that $T = JT^*J$.  
	The general study of such operators was undertaken by the first author, M.~Putinar, and W.~Wogen, 
	in various combinations, in \cite{CSOA, CSO2, SNCSO, CSPI}.  
	A number of other authors have also made significant 
	contributions \cite{CFT, Gilbreath, ZLJ, JKLL, JKLL2,WXH,Tener,Zag,ZLZ}.  

	We consider here the problem of describing all complex symmetric
	weighted composition operators.  Among other results, we produce a class of complex symmetric weighted composition
	operators which includes the hermitian examples obtained in \cite{CowenKo,CGK} as 
	special cases.  We also raise a number of open questions which 
	we hope will spur further research.

\section{Observations and results}

	In what follows, we let $H^2(\beta)$ denote the weighted Hardy space which corresponds to the
	weight sequence $\{ \beta(n) \}_{n=0}^{\infty}$ \cite[Sect.~2.1]{cm}.  For each $w$ in the open unit disk 
	$\D$ and every integer $n \geq 0$, we let 
	$K_w^{(n)}$ denote the unique function in $H^2(\beta)$ which satisfies $\langle f, K_w^{(n)} \rangle = f^{(n)}(w)$ 
	for every $f$ in $H^2(\beta)$.  For convenience, we 
	often choose to write $K_w$ in place of $K_{w}^{(0)}$.		
	If $\phi:\D\to\D$ is analytic, then the \emph{composition operator} $C_{\phi}:H^2(\beta)\to H^2(\beta)$
	is defined by setting 
	\begin{equation*}
		C_{\phi}(f)=f\circ\phi.
	\end{equation*}
	Given another analytic function $\psi:\D\rightarrow\C$,
	we define the \emph{weighted composition operator} $\W$ by setting
	\begin{equation*}
		\W(f)=\psi\cdot(f\circ\phi).
	\end{equation*}
	Assuming that $\W$ is bounded, one has the useful formulae
	\begin{align}
		\W^{\ast}(K_{w}) &=\overline{\psi(w)}K_{\phi(w)}, \label{eq-KernelAdjoint} \\
		\W^{\ast}\bigl(K_{w}^{(1)}\bigr)
		&=\overline{\psi(w)}\,\overline{\phi^{\prime}(w)}
		K_{\phi(w)}^{(1)}+\overline{\psi^{\prime}(w)}K_{\phi(w)}. \label{eq-KernelDerivatives}
	\end{align}

\subsection{Composition operators}
	One initially expects few unweighted composition operators to be
	complex symmetric.  In fact, the only obvious candidates which come to mind are the normal
	composition operators.  These are precisely the operators $C_{\phi}:H^2(\beta)\to H^2(\beta)$ where
	$\phi(z) = az$ and $|a| \leq 1$ \cite[Thm.~8.2]{cm}.  	
	One might initially suspect that these are the \emph{only} complex symmetric composition operators.  
	This na\"{i}ve conjecture proves to be false, however, as there exist at least two other	
	basic families of complex symmetric composition operators.
	
	\begin{Proposition}\label{PropositionInvolutive}
		If $\phi$ is either (i) constant, or (ii) an involutive disk automorphism, 
		then $C_{\phi}: H^2(\beta) \to H^2(\beta)$ is a complex symmetric operator.
	\end{Proposition}
	
	The preceding follows immediately from the fact that an operator which is
	algebraic of degree two is complex symmetric \cite[Thm.~2]{SNCSO}.  In what follows,
	we work only with nonconstant symbols $\phi$.  It turns out that (ii) prompts
	an elementary question whose answer has so far eluded us.    

	\begin{Question}
		Let $\phi$ be an involutive disk automorphism.
		Find an explicit conjugation $J:H^2(\beta)\to H^2(\beta)$ such that $C_{\phi} = JC_{\phi}^*J$.
	\end{Question}
	
	Naturally, one is also interested in determining whether there are any additional classes
	of complex symmetric composition operators.
	
	\begin{Question}
		Characterize all complex symmetric composition operators $C_{\phi}$ on the classical
		Hardy space $H^2$ or, more generally, on weighted Hardy spaces $H^2(\beta)$.
	\end{Question}
	
	In the negative direction, we have the following results.
	
	\begin{Proposition}
		If $C_{\phi}:H^2(\beta) \to H^2(\beta)$ is a hyponormal composition operator which
		is complex symmetric, then $\phi(z) = a z$ where $|a| \leq 1$.
	\end{Proposition}

	\begin{proof}
		Suppose $C_{\phi}$ is hyponormal; that is, $\norm{C_{\phi}f} \geq \norm{C^*_{\phi}f}$ for all $f$ in $H^2(\beta)$. 
		If $C_{\phi}$ is $J$-symmetric, then it follows that
		\begin{equation*}
			\norm{C^*_{\phi}f } = \norm{JC_{\phi} Jf} = \norm{C_{\phi} Jf} \geq \norm{C^*_{\phi} Jf} 
			= \norm{JC_{\phi}f} =\norm{C_{\phi}f}.
		\end{equation*}
		Thus $\norm{C_{\phi}f} = \norm{C^*_{\phi}f}$ for all $f$ in $H^2$
		whence $C_{\phi}$ is normal.  By \cite[Thm.~8.2]{cm} we conclude that $\phi(z) = az$
		where $|a| \leq 1$.
	\end{proof}

	\begin{Proposition}\label{Proposition-J1}
		Suppose that $C_{\phi}:H^2(\beta)\to H^2(\beta)$ is $J$-symmetric.  If $J(1)$ is a constant multiple of a
		kernel function $K_{w}$, then $\phi(w)=w$.  The converse holds whenever $\phi$ is not an automorphism.
	\end{Proposition}

	\begin{proof}
		If $J(1)=\gamma K_{w}$ for some constant $\gamma\neq 0$ and $C_{\phi}$ is $J$-symmetric, then
		\begin{equation*}
			\gamma K_{w} = J(1) = JC_{\phi}(1)=C_{\phi}^{\ast}J(1) 
			=C_{\phi}^{\ast}(\gamma K_{w})=\gamma K_{\phi(w)},
		\end{equation*}
		from which we conclude that $\phi(w) = w$.		
		On the other hand, suppose that $\varphi(w)=w$.  Since $C_{\phi}^*(K_{w})=K_{\phi(w)}=K_{w}$, we see that
		\begin{equation*}
			C_{\phi}J(K_{w})=JC_{\phi}^*(K_{w})=J(K_{w}).
		\end{equation*}
		As long as $\phi$ is not an automorphism, the only eigenvectors for $C_{\phi}$ corresponding to the eigenvalue 
		$1$ are the constant functions \cite[p.~90]{s}.  Therefore $J(K_{w})$ must be a constant function, 
		which means that $J(1)$ must be a  scalar multiple of $K_{w}$.		
	\end{proof}
	
	In light of the preceding, we see that if $J$
	is a conjugation on $H^2(\beta)$ such that $J(1)$ is not a constant multiple of a kernel function,
	then there does not exist a $J$-symmetric composition operator $C_{\phi}$ on $H^2(\beta)$
	whose symbol fixes a point in $\D$.
	If $J(1)$ is a constant multiple of $1$, then we can say 
	even more about $\phi$.  The following
	is inspired by an unpublished result of P.~Bourdon and
	D.~Szajda \cite[Ex.~8.1.2]{cm}.
	
	\begin{Proposition}\label{mono}
		Suppose that $J:H^2(\beta)\to H^2(\beta)$ is a conjugation,
		$J(1)$ is a constant multiple of $1$, and $J(z)$ is a constant
		multiple of $z^{m}$ for some $m\geq 1$. If $C_{\phi}$ is
		$J$-symmetric, then $\phi(z)=az$ for some $|a|\leq 1$.
	\end{Proposition}

	\begin{proof}
		Since $1=\beta(0)K_{0}$, it follows from Proposition \ref{Proposition-J1} that $\phi(0)=0$, whence
		\begin{equation*}
			C_{\phi}^{\ast}\bigl(K_{0}^{(1)}\bigr)
			=\overline{\phi^{\prime}(0)}K^{(1)}_{\phi(0)}
			=\overline{\phi^{\prime}(0)}K^{(1)}_{0}
		\end{equation*}
		by \eqref{eq-KernelDerivatives}.
		Thus $z=\beta(1)K_{0}^{(1)}$ is an eigenvector for
		$C_{\phi}^{\ast}$ corresponding to the eigenvalue
		$\overline{\phi^{\prime}(0)}$. Since $C_{\phi}$ is $J$-symmetric,
		$z^{m}$ must be an eigenvector for $C_{\phi}$ corresponding to the eigenvalue
		$\phi^{\prime}(0)$. Observe that $C_{\phi}(z^{m})=\phi^{m}$, which
		means that $\phi(z)^{m}=\phi^{\prime}(0)z^{m}$. Consequently
		$\phi(z)=az$, where $|a|\leq 1$.
	\end{proof}



\subsection{Weighted composition operators}

	Although our list of complex symmetric composition operators is somewhat sparse,
	there are a variety of \emph{weighted} composition operators 
	which are known to be complex symmetric.  Indeed, the study of hermitian, normal, and unitary weighted composition operators
	has been the focus of intense research \cite{CowenKo, BourdonNarayan, CGK}.  The following 
	is a generalization of \cite[Lem.~2, Prop.~3]{BourdonNarayan}, where the same conclusion
	is obtained under the assumption that $\W$ is normal.

	\begin{Proposition}
		If $\W:H^2(\beta) \to H^2(\beta)$ is complex symmetric, then either
		$\psi$ is identically zero or $\psi$ is nonvanishing on $\D$.  Moreover,
		if $\phi$ is not a constant function and $\psi$ is not identically zero, then $\phi$ is univalent.
	\end{Proposition}

	\begin{proof}
		Suppose that $\W$ is complex symmetric and that $\psi$ does not vanish identically.
		Since $\ker \W = \{0\}$, we conclude that $\ker \W^* = \{0\}$ by \cite[Prop.~1]{CSOA}.
		If $\psi(w)=0$ for some $w$ in $\D$, then $\W^{\ast}(K_{w})=0$ by \eqref{eq-KernelAdjoint}.
		Since this contradicts the fact that $\ker \W^*$ is trivial, we conclude that $\psi$ is nonvanishing on $\D$.
		Now suppose that there are points $w_{1}$ and $w_{2}$ in $\D$ such
		that $\phi(w_{1})=\phi(w_{2})$.  It follows that
		\begin{equation*}
			W^{\ast}\bigl(\overline{\psi(w_{2})}
			K_{w_{1}}-\overline{\psi(w_{1})}K_{w_{2}}\bigr)
			=\overline{\psi(w_{2})}\,\overline{\psi(w_{1})}K_{\phi(w_{1})}-
			\overline{\psi(w_{1})}\,\overline{\psi(w_{2})}K_{\phi(w_{2})}=0\text{.}
		\end{equation*}
		Since any distinct pair of reproducing kernel functions is linearly
		independent, we conclude that $w_{1}=w_{2}$.  In other words,
		$\phi$ is univalent.
	\end{proof}

	The following result provides a severe restriction on the spectrum of a complex symmetric weighted composition operator
	whose symbol has a fixed point in $\D$.

	\begin{Proposition}\label{PropositionEigen}
		Suppose that $\W:H^2(\beta)\to H^2(\beta)$ is a complex symmetric operator.
		If $\phi(w_{0})=w_{0}$ for some $w_{0}$ in $\D$, then
		$\psi(w_{0})\,\phi^{\prime}(w_{0})^{n}$ is an eigenvalue of $\W$
		for every integer $n\geq 0$.
	\end{Proposition}

\begin{proof}
	Since $\W$ is complex symmetric, by \cite[Prop.~1]{CSOA} it suffices to prove that 
	\begin{equation}\label{eq-Ugly}
		\overline{\psi(w_{0})}\,\overline{\phi^{\prime}(w_{0})^{n}}
	\end{equation}
	is an eigenvalue for $\W^{\ast}$.  Let us first assume that  $\phi^{\prime}(w_{0})$ is not a root of unity. 
	We claim that for each $n\geq 0$, the function $K_{w_{0}}^{(n)}$ can be written in the form
	\begin{equation*}
		v_{n}+\alpha_{n-1}v_{n-1}+\alpha_{n-2}v_{n-2} +\cdots+\alpha_{0}v_{0},
	\end{equation*}
	where $v_{j}$ is an eigenvector for $\W$ corresponding to the eigenvalue 
	$\overline{\psi(w_{0})}\,\overline{\phi^{\prime}(w_{0})^{j}}$. We prove this assertion by induction.  Note that
	\begin{equation*}
		\W^{\ast}(K_{w_{0}})=\overline{\psi(w_{0})}K_{\phi(w_{0})}
		=\overline{\psi(w_{0})}K_{w_{0}},
	\end{equation*}
	so the claim holds when $n=0$. Suppose then that the claim holds for all $n\leq k$ and
	consider the kernel function $K_{w_{0}}^{(k+1)}$.  Now recall that $\W^{\ast}\bigl(K_{w_{0}}^{(k+1)}\bigr)$ equals
	$\overline{\psi(w_{0})}\,\overline{\phi^{\prime}(w_{0})^{k+1}}K_{\phi(w_{0})}^{(k+1)}$ plus a linear combination of
	kernel functions $K_{w_{0}}^{(j)}$ with $j\leq k$. Our induction hypothesis implies that each of these kernel functions is a linear
	combination of eigenvectors $v_{j}$. Therefore we may write
	\begin{equation*}
		\W^{\ast}\bigl(K_{w_{0}}^{(k+1)}\bigr)
		=\overline{\psi(w_{0})}\,\overline{\phi^{\prime}(w_{0})
		^{k+1}}K_{w_{0}}^{(k+1)}+
		\beta_{k}v_{k}+\beta_{k-1}v_{k-1}+\cdots+\beta_{0}v_{0}
	\end{equation*}
	for some constants $\beta_{0}$, $\beta_{1}$, \ldots, $\beta_{k}$.  Observe that the function
	\begin{equation*}
		v_{k+1}=K_{w_{0}}^{(k+1)}+\sum_{j=0}^{k}\frac{\beta_{j}}
		{\overline{\psi(w_{0})}\bigl(\overline{\phi^{\prime}(w_{0})^{k+1}}
		-\overline{\phi^{\prime}(w_{0})^{j}}\bigr)}\,v_{j}
	\end{equation*}
	is an eigenvector for $\W^{\ast}$ corresponding to the eigenvalue $\overline{\psi(w_{0})}\,\overline{\phi^{\prime}(w_{0})^{k+1}}$.
	Consequently our claim holds for all $n$.  In other words, every term \eqref{eq-Ugly} is
	an eigenvalue for $\W^{\ast}$.
	If $\phi^{\prime}(w_{0})$ is an $m$th root of unity, then a similar argument shows that
	\begin{equation*}
		K_{w_{0}}^{(n)}=v_{n}+\alpha_{n-1}v_{n-1}+\alpha_{n-2}v_{n-2} +\cdots+\alpha_{0}v_{0}
	\end{equation*}
	whenever $0\leq n\leq m-1$.  Hence \eqref{eq-Ugly} is
	an eigenvalue for $\W^{\ast}$ when $n\leq m-1$ and hence for all $n$.
	In either case, every number \eqref{eq-Ugly} is an eigenvalue for $\W^{\ast}$, which means that
	$\psi(w_{0})\,\phi^{\prime}(w_{0})^{n}$ is an eigenvalue for $\W$.
\end{proof}

\begin{Example}
	Fix $a \in \D \backslash\{0\}$ and let 
	\begin{equation*}
		\phi = \frac{a-z}{1-\overline{a}z}.
	\end{equation*}
	Since $\phi$ is an involutive automorphism, the composition operator $C_{\phi}:H^2(\beta) \to H^2(\beta)$ is complex
	symmetric by Proposition \ref{PropositionInvolutive}.  Moreover, observe that the spectrum
	$\sigma( C_{\phi})$ of $C_{\phi}$ is precisely $\{-1,1\}$.  On the other hand, 
	Proposition \ref{PropositionEigen} implies that $\phi'(w_0)^n$ belongs to $\sigma(C_{\phi})$ whenever $w_0$ is a fixed
	point of $w_0$. However, the only fixed point of $\phi$ which lies inside of $\D$ is 
	\begin{equation*}
		w_0 = \frac{1- \sqrt{1 -|a|^2}}{ \overline{a}},
	\end{equation*}
	which happens to satisfy $\phi'(w_0) = -1$, in accordance with Proposition \ref{PropositionEigen}.
\end{Example}

\subsection{Koenigs eigenfunctions}
	For any nonconstant non-automorphism $\phi:\D\rightarrow\D$ which has a fixed point $w_{0}$ in $\D$
	and for which $\phi'(w_0) \neq 0$, there is an analytic $\kappa:\D\rightarrow\C$ such that
	$\kappa\circ\phi=\phi^{\prime}(w_{0})\kappa$. This function,
	called the \textit{Koenigs eigenfunction} for $\phi$, is unique
	up to scalar multiplication \cite[p.~62, p.~93]{cm}.  Furthermore, $\kappa^{n}$ (or any
	constant multiple thereof) is the only analytic function for which
	$\kappa^{n}\circ\phi=\phi^{\prime}(w_{0})^{n}\kappa^{n}$.
	Proposition \ref{PropositionEigen}, together with the details of its proof,
	yields the following result pertaining to unweighted composition operators.

	\begin{Proposition}\label{PropositionKoenigs}
		Let $\phi:\D\rightarrow\D$ be an analytic selfmap which is not an automorphism and suppose that
		$\phi(w_{0})=w_{0}$ and $\phi'(w_0) \neq 0 $
		for some $w_{0}$ in $\D$.  If $C_{\phi}:H^2(\beta)\to H^2(\beta)$ is complex symmetric, then
		every power $\kappa^{n}$ of the Koenigs eigenfunction for $\phi$ belongs to $H^2(\beta)$.  
	\end{Proposition}

	It is not difficult to construct a univalent map
	$\phi:\D\rightarrow\D$ in such a way that one can readily
	determine whether its Koenigs eigenfunction belongs to $H^2(\beta)$ \cite[pp.~93-94]{s}.  Let
	$\kappa:\D\rightarrow\C$ be a univalent function that vanishes at
	some point $w_{0}$ and consider the region $\kappa(\D)$.  Suppose
	that $\lambda\kappa(\D)\subseteq\kappa(\D)$ for some complex
	$\lambda$ with $|\lambda|<1$.  Define the map $\phi$ by
	$\phi(z)=\kappa^{-1}(\lambda\kappa(z))$.   Then, by construction,
	$\phi$ is a univalent self-map of $\D$ that fixes $w_{0}$ and
	whose Koenigs eigenfunction is $\kappa$.  Hence, by starting with a
	$\kappa$ that belongs to $H^2(\beta)$, we construct a $\phi$ whose
	Koenigs function belongs to $H^2(\beta)$.  Similarly, if we take $\kappa$
	such that $\kappa^{n}$ does not belong to $H^2(\beta)$ for some $n$, we
	obtain a map whose corresponding composition operator is not complex symmetric by 
	Proposition \ref{PropositionKoenigs}.  For example, consider any such $\lambda$ and
	take $\kappa(z)=2z/(1-z)$, which does not belong to the Hardy space
	$H^{2}$.  From this we obtain the map $\phi(z)=(\lambda z)/(1+(\lambda-1)z)$, which induces 
	a composition operator $C_{\phi}:H^2 \to H^2$ which is not complex symmetric.

	Much work has been done to determine the conditions under which a
	Koenigs eigenfunction $\kappa$ belongs to the Hardy space $H^{2}$.
	In this context, Proposition \ref{PropositionKoenigs} is equivalent to saying
	that $\kappa$ belongs to $H^{p}$ for every $0<p<\infty$. The
	following proposition follows directly from \cite[Thm.~2.2]{p}.

	\begin{Proposition}
		Suppose that $\phi:\D\rightarrow\D$ is not an automorphism and
		that $\phi$ has a fixed point $w_0$ in $\D$ such that $\phi'(w_0) \neq 0$.  If $C_{\phi}:H^2\to H^2$ is
		complex symmetric, then the essential spectral radius of $C_{\phi}$ is $0$. In other words,
		$C_{\phi}$ must be a Riesz composition operator.
	\end{Proposition}

	A good deal of work has been done to study Riesz composition
	operators on $H^{2}$.  Bourdon and Shapiro's paper \cite{bs} serves
	as an excellent starting point.  

	Suppose that $\phi$ is not an automorphism,
	$\phi(w_{0})=w_{0}$, $\phi'(w_0) \neq 0$, and that $C_{\phi}$ is $J$-symmetric.  As we
	have already observed, $J(1)$ must be a constant multiple of
	$K_{w_{0}}$. Let $\kappa$ denote the Koenigs eigenfunction for
	$\phi$, normalized so that $\|\kappa\|=1$. We also know that
	$J(\kappa)$ equals a constant multiple of $K_{w_{0}}^{(1)}$.  In
	particular, taking into account the norms of these functions, we can
	write
	\begin{equation*}
		J(1)=\frac{\gamma\,\beta(0)\,K_{w_{0}}}{\|K_{w_{0}}\|},\qquad
		J(\kappa)=\frac{\delta\,K_{w_{0}}^{(1)}}{\bigl\|K_{w_{0}}^{(1)}\bigr\|},
	\end{equation*}
	where $|\gamma|=|\delta|=1$. Since $\langle\kappa,1\rangle=\langle J(1),J(\kappa)\rangle$, we see that
	\begin{equation*}
		|\kappa(0)|=\frac{\bigl|K_{w_{0}}^{(1)}(w_{0})\bigr|}
		{\|K_{w_{0}}\|\bigl\|K_{w_{0}}^{(1)}\bigr\|}.
	\end{equation*}
	If $w_{0}=0$, then this tells us nothing.  If $w_{0}\neq 0$, however, it
	places a major restriction upon the function $\kappa$.  In essence,
	most functions in $H^2(\beta)$ cannot be Koenigs eigenfunctions for
	complex symmetric composition operators.  


\subsection{An instructive example}
	We conclude this note by producing a class of complex symmetric weighted composition
	operators which includes the hermitian examples obtained in \cite{CowenKo,CGK} as 
	special cases.  For each $\kappa \geq 1$, let $H^2(\beta_{\kappa})$ 
	denote the weighted Hardy space whose reproducing kernel is 
	$K_w(z) = (1 - \overline{w}z)^{-\kappa}$.
	We now explicitly characterize all weighted composition operators on $H^2(\beta_{\kappa})$ 
	which are $J$-symmetric with respect to the conjugation
	\begin{equation}\label{eq-standard}
		[Jf](z) = \overline{f(\overline{z})}
	\end{equation}
	on $H^2(\beta_{\kappa})$.  For the sake of convenience, we sometimes write $\widetilde{f}:=Jf$.

	\begin{Proposition}\label{PropositionFormula}
		A weighted composition operator $\W:H^2(\beta_{\kappa})\to H^2(\beta_{\kappa})$
		is $J$-symmetric with respect to the conjugation \eqref{eq-standard} if and only if
		\begin{equation}\label{eq-PPF}
			\psi(z) = \frac{b}{(1 - a_0 z)^{\kappa}}, \qquad \phi(z) = a_0 + \frac{a_1 z}{1 - a_0 z},
		\end{equation}
		where $a_0$ and $a_1$ are constants such that $\phi$ maps $\D$ into $\D$.  Moreover,
		such an operator is normal if and only if either,
		\begin{enumerate}\addtolength{\itemsep}{0.5\baselineskip}
			\item $b = 0$,
			\item $b \neq 0$ and $\Im a_0 \overline{a_1} = (1 - |a_0|^2) \Im a_0$.
		\end{enumerate}
		Moreover, $\W$ is hermitian if and only if $a_0$, $a_1$, and $b$ each belong to $\R$.
	\end{Proposition}
	
	\begin{proof}
		To streamline our notation, we let $W:= \W$.
		A simple computation now confirms that if $\psi$ and $\phi$ are given by \eqref{eq-PPF},
		then $W J K_w = JW^*K_w$ for all $w$ in $\D$,  implying that $W = J W^* J$.  On the other hand, if
		$W = JW^*J$, then $W J K_w  = JW^* K_w$ for all $w$ in $\D$.  Since $JK_w = K_{\overline{w}}$, this implies that
		\begin{equation}\label{psiphikernel}
			\psi(z) K_{\overline{w}}(\phi(z)) = \psi(w) K_{\overline{\phi(w)}}(z)
		\end{equation}
		holds for all $z,w$ in $\D$.  Setting $w = 0$ in the preceding we find that
		\begin{equation*}
			\psi(z) = \frac{\psi(0)}{(1 - \phi(0)z)^{\kappa}}.
		\end{equation*}
		Thus $\psi$ is of the form \eqref{eq-PPF} with $b = \psi(0)$ and $a_0 = \phi(0)$.
		From \eqref{psiphikernel} it follows that
		\begin{equation*}
			\frac{1  - \phi(w)z}{1 - a_0 z} = \frac{1 - \phi(z)w}{1 - a_0 w}.
		\end{equation*}
		Writing $\phi(z) = a_0 + z \xi(z)$ where $\xi$ is analytic on $\D$, we see that
		\begin{equation*}
			(1 - a_0 z) \xi(z) = (1 - a_0 w) \xi(w)
		\end{equation*}
		for all $z,w$ in $\D$.  Thus the function $(1 - a_0 z)\xi(z)$ is constant.  Letting
		$\xi(0) = a_1$, we conclude that $\phi$ has the form \eqref{eq-PPF}.
		
		Suppose that $\psi$ and $\phi$ are given by \eqref{eq-PPF} and note that $W$ is normal if and only if
		$JWW^* K_w = WW^*J K_w$ for all $w$ in $\D$.  The preceding condition
		is equivalent to asserting that
		\begin{equation*}
			\frac{\psi(w)\widetilde{\psi}(z)}{1 - \phi(w) \widetilde{\phi}(z)} = 
			\frac{\widetilde{\psi}(w)\psi(z)}{1 - \widetilde{\phi}(w)\phi(z) }
		\end{equation*}
		holds for all $z,w$ in $\D$.  Taking the reciprocal of the preceding and simplifying,
		we see that equality holds for all $z,w$ if and only if either $b=0$ or $b \neq 0$
		and $\Im a_0 \overline{a_1} = (1 - |a_0|^2) \Im a_0$.

		We also note that $W = W^*$ if and only if $W JK_w= J W K_w$, which yields
		\begin{equation*}
			\psi(z)K_{\overline{w}}(\phi(z)) = \widetilde{\psi}(z) K_{\overline{w}}( \widetilde{\phi}(z)).
		\end{equation*}
		Setting $w = 0$ in the preceding yields $\psi(z) = \widetilde{\psi}(z)$ so that $a_0$ and $b$ are real.
		This implies that $\phi(z) = \widetilde{\phi}(z)$ whence $a_1$ is also real.  Conversely, it is
		easy to see that if $a_0$, $a_1$, and $b$ are real, then $W$ is hermitian.
	\end{proof}

	It follows from the preceding that if $a_0,a_1,b$ are chosen so that $\phi$ maps $\D$ into $\D$
	and so that (i) and (ii) both fail to hold, then the operator $\W:H^2(\beta_{\kappa}) \to H^2(\beta_{\kappa})$ will be
	complex symmetric and non-normal.  Moreover, the operators produced by Proposition
	\ref{PropositionFormula} include the hermitian examples considered in \cite{CowenKo, CGK}.
		
	\begin{Question}
		The detailed spectral structure of \emph{hermitian} weighted composition operators 
		$\W:H^2(\beta_{\kappa}) \to H^2(\beta_{\kappa})$ with $\psi$ and $\phi$ given by \eqref{eq-PPF}
		is studied in \cite{CowenKo,CGK}.  What is the corresponding spectral theory for the
		non-normal weighted composition operators arising from Proposition \ref{PropositionFormula}?
	\end{Question}

\bibliography{WWCOCS}

\end{document}